\documentclass[12pt,twoside]{article}
\title{On $(n,d)$-perfect rings}
\date{}
\usepackage{amsfonts}
\usepackage{amsmath}
\usepackage{amssymb}
\usepackage{latexsym}
\newtheorem{thm}{\bf Theorem}[section]
\newtheorem{cor}[thm]{\bf Corollary}
\newtheorem{lem}[thm]{\bf Lemma}
\newtheorem{prop}[thm]{\bf Proposition}
\newtheorem{defn}[thm]{\bf Definition}

\newtheorem{exmp}[thm]{\bf Example}

\catcode`\ç=13
\defç{\c{c}}
\catcode`\é=13
\defé{\'e}
\catcode`\à=13
\defà{\`a}
\catcode`\è=13
\defè{\`e}
\catcode`\â=13
\defâ{\^a}
\catcode`\ù=13
\defù{\`u}
\catcode`\ê=13
\defê{\^e}
\catcode`\î=13
\defî{\^\i}
\catcode`\ô=13
\defô{\^o}
\newcommand{\field}[1]{\mathbb{#1}}

\newcommand{\Z }{\field{Z}}

\def\wdim{{\rm wdim}}

\def\gldim{{\rm gldim}}

\def\pd{{\rm pd}}

\def\fd{{\rm fd}}

\def\Coker{{\rm Coker}}

\def\sup{{\rm sup}}

\def\qf{{\rm qf}}

\def\Ext{{\rm Ext}}

\def\Tor{{\rm Tor}}

\begin{document}
\thispagestyle{empty}
%%%%%%%%%%%%%%%%%%%%%%%%%%%%%%%%%%%%%%%%%%%%%%%%%%%%%%%%%
%%%%%%%%%%%%%%%%%%%%%%%%%%%%%%%%%%%%%%%%%%%%%%%%%%%%%%%%%
%%%%%%%%%%%%%%%%%%%%%%%%%%%%%%%%%%%%%%%%%%%%%%%%%%%%%%%%%
%%%TITLE%%%%%%%%%%%%%%%%%%%%%%%%%%%%%%%%%%%%%%%%%%%%%%%%%
\maketitle \vspace*{-2cm}
\begin{center}{\large\bf Abdellatif Jhilal and Najib Mahdou}
%%%%%%%%%%%%%%%%%%%%%%%%%%%%%%%%%%%%%%%%%%%%%%%%%%%%%%%%%
%%%%%%%%%%%%%%%%%%%%%%%%%%%%%%%%%%%%%%%%%%%%%%%%%%%%%%%%%
%%%%%%%%%%%%%%%%%%%%%%%%%%%%%%%%%%%%%%%%%%%%%%%%%%%%%%%%%
%%%NAMES%%%%%%%%%%%%%%%%%%%%%%%%%%%%%%%%%%%%%%%%%%%%%%%%%

\bigskip
%%%%%%%%%%%%%%%%%%%%%%%%%%%%%%%%%%%%%%%%%%%%%%%%%%%%%%%%%
%%%%%%%%%%%%%%%%%%%%%%%%%%%%%%%%%%%%%%%%%%%%%%%%%%%%%%%%%
%%%%%%%%%%%%%%%%%%%%%%%%%%%%%%%%%%%%%%%%%%%%%%%%%%%%%%%%%
%%%%%%%%%%%%ADDRESSES%%%%%%%%%%%%%%%%%%%%%%%%%%%%%%%%%%%%%%%%%%%%%
\small{Department of Mathematics, Faculty of Science and
Technology of Fez,\\ Box 2202, University S. M.
Ben Abdellah Fez, Morocco, \\  Jhilalabdo@hotmail.com \\
mahdou@hotmail.com}
\end{center}\bigskip\bigskip
%%%%%%%%%%%%%%%%%%%%%%%%%%%%%%%%%%%%%%%%%%%%%%%%%%%%%%%%%
%%%%%%%%%%%%%%%%%%%%%%%%%%%%%%%%%%%%%%%%%%%%%%%%%%%%%%%%%
%%%%%%%%%%%%%%%%%%%%%%%%%%%%%%%%%%%%%%%%%%%%%%%%%%%%%%%%%
%%%%%%%%%%%%%%%%%%%%%%%%%%%%%%%%%%%%%%%%%%%%%%%%%%%%%%%%%%
%%%%%%%%%%%%%%%%%%%%%%%%%%%%%%%%%%%%%%%%%%%%%%%%%%%%%%%%%
%%%ABSTRACT%%%%%%%%%%%%%%%%%%%%%%%%%%%%%%%%%%%%%%%%%%%%%%

%%%%%%%%%%%%%%%%%%%%%%%%%%%%%%%%%%%%%%%%%%%%%%%%%%%%%%%%%
%%%%%%%%%%%%%%%%%%%%%%%%%%%%%%%%%%%%%%%%%%%%%%%%%%%%%%%%%
%%%%%%%%%%%%%%%%%%%%%%%%%%%%%%%%%%%%%%%%%%%%%%%%%%%%%%%%%
\abstract{In this paper, we introduce the notion of
``$(n,d)$-perfect rings'' which is in some way a generalization of
the notion of ``$S$-rings''. After we give some basic results of
this rings and we survey the relationship between ``$A(n)$
property'' and ``$(n,d)$-perfect property''. Finally, we
investigate the ``$(n,d)$-perfect property'' in pullback rings.}

\small{\noindent{\bf Key Words.}$(n,d)$-perfect ring, $A(n)$ ring,
$n$-presented, homological dimensions, pullback ring.

%%%%%%%%%%%%%%%%%%%%%%%%%%%%%%%%%%%%%%%%%%%%%%%%%%%%%%%%%%%%%%%%%%%%%%%%%%%%%%%%%%%%%%%%%%%%%%%%%%%%%%%%%%%%%%%%%%%%
%%%%%%%%%%%%%%%%%%%%%%%%%%%%%%%%%%%%%%%%%%%%%%%%%%%%%%%%%%%%%%%%%%%%%%%%%%%%%%%%%%%%%%%%%%%%%%%%%%%%%%%%%%%%%%%%%%%%
%%%%%%%%%%%%%%%%%%%%%%%%%%%%%%%%%%%%%%%%%%%%%%%%%%%%%%%%%%%%%%%%%%%%%%%%%%%%%%%%%%%%%%%%%%%%%%%%%%%%%%%%%%%%%%%%%%%%
%%%%%%%%%%%%%%%%%%%%%%%%%%%%%%%%%%%%%%%%%%%%%%%%%%%%%%%%%%%%%%%%%%%%%%%%%%%%%%%%%%%%%%%%%%%%%%%%%%%%%%%%%%%%%%%%%%%%

%%%%%%%%%%%%%%%%%%%%%%%%%%%%%%%%%%%%%%%%%%%%%%%%%%%%%%%%%%%%%%%%%%%%%%%%%%%%%%%%%%%%%%%%%%%%%%%%%%%%%%%%%%%%%%%%%%%%
%%%%%%%%%%%%%%%%%%%%%%%%%%%%%%%%%%%%%%%%%%%%%%%%%%%%%%%%%%%%%%%%%%%%%%%%%%%%%%%%%%%%%%%%%%%%%%%%%%%%%%%%%%%%%%%%%%%%
%%%%%%%%%%%%%%%%%%%%%%%%%%%%%%%%%%%%%%%%%%%%%%%%%%%%%%%%%%%%%%%%%%%%%%%%%%%%%%%%%%%%%%%%%%%%%%%%%%%%%%%%%%%%%%%%%%%%
%%%%%%%%%%%%%%%%%%%%%%%%%%%%%%%%%%%%%%%%%%%%%%%%%%%%%%%%%%%%%%%%%%%%%%%%%%%%%%%%%%%%%%%%%%%%%%%%%%%%%%%%%%%%%%%%%%%%
\begin{section}{Introduction}

  The object of this paper is to introduce a doubly filtered set of
 classes of rings which may serve to shed further light on the
 structures of non-Noetherian rings.
  Throughout this work, all rings are commutative with identity
  element, and all modules are unitary. By a ``local'' ring we mean
  a (not necessarily Noetherian) ring with a unique maximal
  ideal.

The classes of rings we will define here are in some ways
  generalizations of the notion of ``A(n) rings'' which is
  introduced by Cox and Pendleton in \cite{c}.

 Let $R$ be a ring and let $M$ be an $R$-module. As usual we use
$\pd_R(M)$ and $\fd_R(M)$ to denote the usual projective and flat
dimensions of $M$, respectively.  The classical global and weak
dimension of $R$ are respectively $\gldim(R)$ and $\wdim(R)$ . If
$R$ is an integral domain, we denote its quotient field by
$\qf(R)$.
 \\ We say that
$M$ is $n$-presented $R$-module whenever
  there is an exact sequence $F_{n}\longrightarrow
  F_{n-1}\longrightarrow ...\longrightarrow F_{0} \longrightarrow
  M \longrightarrow 0$ of $R$-modules in which each $F_{i}$ is a
  free finitely generated $R$-module. In particular, 0-presented
  and $1$-presented $R$-modules are respectively finitely
  generated and finitely presented $R$-modules. We recall that a
  coherent ring is a ring such that each finitely generated ideal
  is finitely presented.
\\As in \cite{Bo, V}, we set $\lambda_{R}(M)=sup\{n/ M$ is
   $n$-presented$ \}$ except that we set $\lambda_{R}(M)=-1$ if
   $M$ is not finitely generated. Note that $\lambda_{R}(M) \geq
   n$ is a way to express the fact that $M$ is $n$-presented.

     The headings of the various sections indicate their content.
    Thus in the second section, we present the definition of this
    classes of rings and ( mostly we known) the basic results. The
    third section is devoted to establish the relationship between
    the ``$A(n)$ property'' and the property of the classes of
    rings which we will define below, and the fourth section of
    this survey deal with these classes of rings in pullbacks.
    Finally, we give an extensive set of references.
\\ General background materials can be found in Rotman
\cite{Ro} (1979), and Glaz \cite{G} (1989).
\end{section}

%%%%%%%%%%%%%%%%%%%%%%%%%%%%%%%%%%%%%%%%%%%%%%%%%%%%%%%%%%%%%%%%%%%%%%%%%%%%%%%%%%%%%%%%%%%%%%%%%%%%%%%%%%%%%%%%%%%%
%%%%%%%%%%%%%%%%%%%%%%%%%%%%%%%%%%%%%%%%%%%%%%%%%%%%%%%%%%%%%%%%%%%%%%%%%%%%%%%%%%%%%%%%%%%%%%%%%%%%%%%%%%%%%%%%%%%%
%%%%%%%%%%%%%%%%%%%%%%%%%%%%%%%%%%%%%%%%%%%%%%%%%%%%%%%%%%%%%%%%%%%%%%%%%%%%%%%%%%%%%%%%%%%%%%%%%%%%%%%%%%%%%%%%%%%%
%%%%%%%%%%%%%%%%%%%%%%%%%%%%%%%%%%%%%%%%%%%%%%%%%%%%%%%%%%%%%%%%%%%%%%%%%%%%%%%%%%%%%%%%%%%%%%%%%%%%%%%%%%%%%%%%%%%%
\begin{section}{ Definition and
basic results}\label{sec:2}

\quad In this section we introduce and study the $(n,d)$-perfect
 ring which is defined as follows.
\begin{defn}
Let $n$ and $d$ be a nonnegatives integers. A ring $R$ is said to
be an $(n,d)$-perfect ring, if every $n$-presented module with
flat dimension at most $d$, has projective dimension at most $d$.

\end{defn}

  We illustrate this notion with the following example. First it is
  well known that if a flat $R$-module $M$ is finitely presented, or
  finitely generated with $R$ either a semilocal ring or an integral domain,
  then $M$ is projective ( see \cite[Theorem 2]{E}).
  And a ring is called an $S$-ring if every finitely generated flat
  $R$-module is projective (see \cite{P}).

 \begin{exmp}
\begin{enumerate}
\item $R$ is an $S$-ring if and only if $R$ is an $(0,0)$-perfect
ring.
\item If $R$ is a semilocal ring, then $R$ is an $(n,n)$-perfect ring for
every $n\geq 0$.
\item If $R$ is a domain, then $R$ is an $(n,n)$-perfect ring for
every $n\geq 0$.
\item If $R$ is an $(n,d)$-perfect ring, then $R$ is an $(n',d)$-perfect ring for
every $n'\geq n$.
\item For every $n>d$,  $R$ is an $(n,d)$-perfect ring.
\item If $R$ is a perfect ring, then $R$ is  $(n,d)$-perfect for
every $n\geq 0$ an $d\geq0$.

\end{enumerate}
\end{exmp}

\begin{proof} Obvious.  \end{proof}

 The following proposition give two results concerning the
 Noetherian rings and coherent rings.
\begin{prop}
\begin{enumerate}

\item If $R$ is a Noetherian ring, then $R$ is an  $(n,d)$-perfect ring for
every $n\geq 0$ and $d\geq 0$.

\item If $R$ is a coherent ring, then $R$ is an  $(n,d)$-perfect ring for
every $n\geq 1$ and $d\geq 0$.
\end{enumerate}
\end{prop}
\begin{proof} Obvious.\end{proof}

Furthermore, we construct an example of ring which it is a
$(0,1)$- perfect ring and which is not a $(0,0)$-perfect ring
(Example \ref{14}). Also we exhibit an example of a
$(1,1)$-perfect ring which is not a $(0,1)$-perfect ring (Example
\ref{15}).

\begin{exmp}\label{14}
 Let $R$ be an hereditary and Von Neumann regular ring which it is
 not semi-simple. Then  $R$ is a $(0,1)$-perfect ring which is not a
$(0,0)$-perfect ring.

\end{exmp}

\begin{proof} The ring $R$ is a $(0,1)$-perfect ring since $R$
 is an hereditary ring. If $R$ is a $(0,0)$-perfect ring hence
 every finitely generated $R$-module is projective since $R$ is a
 Von Neumann regular ring, hence we have a contradiction with $R$ is
 not semi-simple.
\end{proof}

\begin{exmp}\label{15}
 Let $R$ be a non-Noetherian Pr\"ufer domain. Then $R$ is a
$(1,1)$-perfect domain which it is not a $(0,1)$-perfect domain.
 \end{exmp}

\begin{proof} The ring $R$ is a $(1,1)$-perfect ring since $R$ is a domain.
 On the other hand, we show that $R$ is not a $(0,1)$-perfect ring.
 Let $I$ be a not finitely generated ideal of $R$ (since $R$ is not Noetherian), then $I$ is not
 projective. Of course $I$ is flat since $\wdim(R)\leq 1$. Thus
 $R/I$ is $0$-presented with $\fd_{R}(R/I) \leq 1$ and $\pd_{R}(R/I)\geq
 2$ as desired. \end{proof}

 Next we give a homological characterization of $(n,d)$-perfect
 ring.

 \begin{thm}\label{3}
Let $R$ be a commutative ring. Then the following are equivalent.
\begin{enumerate}
\item $R$ is an $(n, d)$-perfect ring.
\item $\Ext^{d+1}_{R}(M,N)= 0$ for all $R$-modules $M$, $N$ such
that $\lambda_{R}(M)\geq n$, \\$\fd_{R}(M)\leq d$ and
$\fd_{R}(N)\leq d$.

\item $\Ext^{d+1}_{R}(M,N)= 0$ for all $R$-modules $M$, $N$ such
that $\lambda_{R}(M)\geq n$,\\ $\lambda_{R}(N)\geq n-(d+1)$,
$\fd_{R}(M) \leq d$ and $\fd_{R}(N)\leq d$.

\end{enumerate}

 \end{thm}\bigskip

The proof of this Theorem involves the following Lemmas.
\begin{lem}\label{1}
Let $R$ be a ring, and let $M$ be an  $n$-presented flat
$R$-module, where $n\geq 0$. Then  $M$ is projective if and only
if $\Ext^{1}_{R}(M,N)=0$ for all $R$-modules $N$ such that
$\lambda_{R}(N)\geq n-1$ and $N$ is a flat $R$-module.

\end{lem}
\begin{proof} Necessity is clear. To proof sufficiency, let $0\longrightarrow K \longrightarrow F \longrightarrow M
\longrightarrow 0$  be an exact sequence with $F$ a finitely
generated free $R$-module. Then $K$ is an $(n-1)$-presented flat
$R$-module, hence by hypothesis $\Ext^{1}_{R}(M,K)=0$ . It follows
that the exact sequence splits, making $M$ a direct summand of
$F$. Therefore $M$ is a projective $R$-module.\end{proof}

\begin{lem} \label{17} Let $R$ be a ring, and let $M$ be an $n$-presented
$R$-module such that $\fd_{R}(M)\leq d$. Then $pd_{R}(M)\leq d$ if
and only if $\Ext^{d+1}_{R}(M,N)=0$ for all $R$-modules $N$ such
that  $\lambda_{R}(N)\geq n-(d+1)$ and $\fd_{R}(N)\leq d$.
\end{lem}
\begin{proof} This follows from Lemma \ref{1} by dimension shifting.
\end{proof}

\begin{proof}[Proof of Theorem \ref{3}] Follows from Lemma
\ref{17}.\end{proof}

Now, in the following we prove that the $(n,d)$-perfect property
descends into a faithfully flat ring homomorphism.\bigskip

\begin{thm}\label{2}

Let $R\longrightarrow S$ be a ring homomorphism making $S$ a
faithfully flat $R$-module. If $S$ is an $(n,d)$-perfect ring then
$R$ is an $(n,d)$-perfect ring.

\end{thm}
\begin{proof} Let $M$ be an $n$-presented  $R$-module with $\fd_{R}(M)\leq
d$. Our aim is to show that $\pd_{R}(M)\leq d$. We have $
\lambda_{S}(M\otimes_{R}S)\geq n$ and $\fd_{S}(M\otimes_{R}S)\leq
d$ since $S$ is a flat $R$-module, so $\pd_{S}(M\otimes_{R}S)\leq
d$ since $S$ is an $(n, d)$-perfect ring.
\\ Let $0\longrightarrow P \longrightarrow F_{d-1}\longrightarrow
...\longrightarrow F_{1}\longrightarrow F_{0}\longrightarrow M
\longrightarrow 0$ be an exact sequence of $R$-modules, where
$F_{i}$ is a free $R$-module for each $i$ and $P$ is a flat
$A$-module. Thus  $P\otimes_{R}S$ is a projective $S$-module.
 By \cite[Example 3.1.4. page 82]{Gr} $P$ is a projective $R$-module  .\end{proof}

  We use  this result  to study the ``$(n,d)$-perfect
property''  of some particular rings.

\begin{cor}
\begin{enumerate}
\item Let $A\subset B$ be two rings such that $B$ is a flat
$A$-module. Let $S=A+XB[X]$, where X is an indeterminate over $B$
. If $S$ is an $(n,d)$-perfect ring, then so is $A$ .
\item Let $R$ be a ring and $X$ is an indeterminate over $R$. If $R[X]$ is
 an $(n,d)$-perfect ring  then so is $R$ .

\end{enumerate}

\end{cor}

\begin{proof}  1) The ring $B$ is a flat $A$-module and $XB[X]\cong B[X]$
thus $S= A + XB[X]$ is a faithfully flat $A$-module. By Theorem
\ref{2} the ring $A$ is an $(n,d)$-perfect  since $S$ is an
$(n,d)$-perfect ring.
\\ 2) Similarly, by Theorem \ref{2}, since $R[X]$  is a faithfully flat
$R$-module.\end{proof}

 We close this section by establishing the
 transfer of the $(n,d)$-perfect property to finite direct
 product.

\begin{thm}\label{26} Let $(R_{i})_{i=1,...,m}$ be a family of rings. Then
$\prod_{i=1}^{m}R_{i}$ is an $(n,d)$-perfect ring if and only if
$R_{i}$ is an $(n,d)$-perfect ring for each $i=1,...,m$ .

\end{thm}\bigskip

 The proof of this Theorem involves the following results.
 \begin{lem}\cite[Lemma 2.5.]{M}\label{9}  Let $(R_{i})_{i=1,2}$ be a family of rings and $E_{i}$
 be an $R_{i}$-module for $i=1,2$. We have :
\begin{enumerate}

 \item $\pd_{R_{1}\times R_{2}}(E_{1}\times
E_{2})=\sup\{\pd_{R_{1}}(E_{1}),\pd_{R_{2}}(E_{2})\}$.
 \item $\lambda_{R_{1}\times R_{2}}(E_{1}\times E_{2})= \inf
\{\lambda_{R_{1}}(E_{1}), \lambda_{R_{2}}(E_{2})\}$

\end{enumerate}

 \end{lem}

 \begin{lem} \label{10} Let $(R_{i})_{i=1,2}$ be a family of rings and $E_{i}$
 be an $R_{i}$-module for $i=1,2$. We have : $\fd_{R_{1}\times R_{2}}(E_{1}\times
 E_{2})=\sup\{\fd_{R_{1}}(E_{1}),\fd_{R_{2}}(E_{2})\}$.
 \end{lem}\bigskip
 \begin{proof} This proof  is analogous to the proof of Lemma
 \ref{9} $(1)$.\end{proof}
\begin{proof}[Proof of Theorem \ref{26}]  We use  induction on m, it suffices to prove the
 assertion for $m=2$. Let $R_{1}$ and $R_{2}$ be two rings such that
 $R_{1}\times R_{2}$ is an $(n,d)$-perfect ring.
 \\ Let $E_{1}$ be an $R_{1}$-module such that $\fd_{R_{1}}(E_{1})\leq
 d$, $\lambda_{R_{1}}(E_{1})\geq n$
 and let $E_{2}$ be an $R_{2}$-module such that $\fd_{R_{2}}(E_{2})\leq
 d$, $\lambda_{R_{2}}(E_{2})\geq n$.\\ By Lemma
 \ref{10}.
$\fd_{R_{1}\times R_{2}}(E_{1}\times
 E_{2})=\sup\{\fd_{R_{1}}(E_{1}),\fd_{R_{2}}(E_{2})\}$.
  And  by Lemma
 \ref{9}(2)  $\lambda_{R_{1}\times R_{2}}(E_{1}\times E_{2})= \inf
\{\lambda_{R_{1}}(E_{1}), \lambda_{R_{2}}(E_{2})\}$ . Thus
$\lambda_{R_{1}\times R_{2}}(E_{1}\times E_{2})\geq n$ and
$\fd_{R_{1}\times R_{2}}(E_{1}\times
 E_{2}) \leq d$. So
$\pd_{R_{1}\times R_{2}}(E_{1}\times
 E_{2})\leq d \quad$   since   $ \quad R_{1}\times R_{2}$ is an $(n,d)$-perfect ring. By
 Lemma \ref{9}(1) $\pd_{R_{1}\times R_{2}}(E_{1}\times
 E_{2})=\sup\{\pd_{R_{1}}(E_{1}),\pd_{R_{2}}(E_{2})\}$ .
 \\ Thus $\pd_{R_{1}}(E_{1})\leq n$ and
$\pd_{R_{2}}(E_{2})\leq n$. Therefore $R_{1}$ is an
$(n,d)$-perfect ring and $R_{2}$ is an $(n,d)$-perfect ring.
\\

Conversely, let $R_{1}$ and $ R_{2}$ be  two $(n,d)$-perfect rings
and let $E_{1}\times E_{2}$ be an $R_{1}\times R_{2}$-module where
$E_{i}$ is an $R_{i} $-module for each $i=1,2 $, such that
$\fd_{R_{1}\times R_{2}}(E_{1}\times
 E_{2})\leq d $ and  $\lambda_{R_{1}\times R_{2}}(E_{1}\times
 E_{2})\geq n$.
 By Lemma \ref{9},  $\lambda_{R_{1}}(E_{1})\geq n$,
$\lambda_{R_{2}}(E_{2}) \geq n$ and by Lemma \ref{10},
$\fd_{R_{1}}(E_{1})\leq d$, $\fd_{R_{2}}(E_{2})\leq d$,  then
$\pd_{R_{1}}(E_{1})\leq d$ and $\pd_{R_{2}}(E_{2})\leq d$, since
$R_{1}$ and $R_{2}$ are an $(n,d)$-perfect rings . By Lemma
\ref{9}, $\pd_{R_{1}\times R_{2}}(E_{1}\times
 E_{2})\leq d $. Therefore $R_{1}\times R_{2}$ is an $(n,d)$-perfect
 rings \end{proof}

\end{section}
%%%%%%%%%%%%%%%%%%%%%%%%%%%%%%%%%%%%%%%%%%%%%%%%%%%%%%%%%%%%%%%%%%%%%%%%%%%%%%%%%%%%%%%%%%%%%%%%%%%%%%%%%%%%%%%%%%%%
%%%%%%%%%%%%%%%%%%%%%%%%%%%%%%%%%%%%%%%%%%%%%%%%%%%%%%%%%%%%%%%%%%%%%%%%%%%%%%%%%%%%%%%%%%%%%%%%%%%%%%%%%%%%%%%%%%%%
%%%%%%%%%%%%%%%%%%%%%%%%%%%%%%%%%%%%%%%%%%%%%%%%%%%%%%%%%%%%%%%%%%%%%%%%%%%%%%%%%%%%%%%%%%%%%%%%%%%%%%%%%%%%%%%%%%%%
%%%%%%%%%%%%%%%%%%%%%%%%%%%%%%%%%%%%%%%%%%%%%%%%%%%%%%%%%%%%%%%%%%%%%%%%%%%%%%%%%%%%%%%%%%%%%%%%%%%%%%%%%%%%%%%%%%%%

%%%%%%%%%%%%%%%%%%%%%%%%%%%%%%%%%%%%%%%%%%%%%%%%%%%%%%%%%%%%%%%%%%%%%%
%%%%%%%%%%%%%%%%%%%%%%%%%%%%%%%%%%%%%%%%%%%%%%%%%%%%%%%%%%%%%%%%%%%%%%
%%%%%%%%%%%%%%%%%%%%%%%%%%%%%%%%%%%%%%%%%%%%%%%%%%%%%%%%%%%%%%%%%%%%%%%%%%%%%%%%%%%%%%%
%%%%%%%%%%%%%%%%%%%%%%%%%%%%%%%%%%%%%%%%%%%%%%%%%%%%%%%%%%%%%%%%%%%%%%

\begin{section}{Relationship between the $A(n)$ property and
$(n,d)$-perfect property} The purpose of the present section is to
establish a natural bridge between $A(n)$ rings and
$(n,d)$-perfect rings.
\\ First, we recall the definition of the $A(n)$ rings introduced in \cite{c}.
\begin{defn}\cite[page 139]{c}
 Let $n$ be nonnegative integer. A ring $R$ is said to be an
 $A(n)$ ring if given any exact sequence $0\longrightarrow M
 \longrightarrow E_{1} \longrightarrow ...\longrightarrow E_{n} $
 of finitely generated $R$-modules with $M$ flat and $E_{i}$ free
 for each $i$, then $M$ is projective.
\end{defn}

We show in the next Theorem  the existence of relationship between
the ``$A(n)$ property'' and ``$(n,d)$-perfect property''.

\begin{thm}\label{12}
 A ring $R$ is an $A(n)$ ring if and only if $R$ is an
 $(n,n)$-perfect ring.
\end{thm}
\begin{proof} Assume that $R$ is an $A(n)$ ring and let $M$ be an $R$-module such that
$\lambda_{R}(M)\geq n$ and $\fd_{R}(M)\leq n$. Then there exist an
exact sequence $0 \longrightarrow P \longrightarrow F_{n-1}
\longrightarrow ... \longrightarrow F_{0} \longrightarrow M
\longrightarrow 0$ of finitely generated $R$-modules with $P$ flat
and $F_{i} $ free for each $i$. Then $P$ is projective since $R$
is an $A(n)$ ring. Therefore $R$ is an $(n,n)$-perfect ring.

Conversely, assume that $R$ is an $(n,n)$-perfect ring. Let
$0\longrightarrow M \longrightarrow F_{1}\longrightarrow
...\stackrel{u_{n}}\longrightarrow F_{n}$
 be an exact sequence of finitely generated $R$-modules with $M$
 flat and $F_{i}$ free for each $i$. We show that $M$ is projective. The exact
sequence $$0\longrightarrow M \longrightarrow F_{1}\longrightarrow
...\stackrel{u_{n}}\longrightarrow F_{n}\longrightarrow \Coker
u_{n} \longrightarrow 0$$ show that  $\lambda_{R}(\Coker u_{n})
\geq n$ and $ \fd_{R}(\Coker u_{n}) \leq n$. Hence,
$\pd_{R}(\Coker u_{n})\leq n$ since $R$ is an $(n,n)$-perfect ring
and so $M$ is projective . Therefore $R$ is an $A(n)$ ring.
\end{proof}

The  Theorem \ref{12} combined with the results obtained by Cox
and Pendleton \cite{c} may be used to find several corollaries.

\begin{cor}
 Let $ \varphi: R\hookrightarrow T$ be an injective ring
 homomorphism.
 \begin{enumerate}
  \item If $T$ is a  $(0,0)$-perfect ring so is $R$.
  \item If $T$ is a  $(n,n)$-perfect ring $(n\geq1)$ and T is a
  flat $R$-module then $R$ is an $(n,n)$-perfect  ring.

\end{enumerate}

\end{cor}
\begin{proof} By \cite[Theorem 2.4]{c} and Theorem \ref{12}. \end{proof}

\begin{thm}
A ring $R$ is a $(1,1)$-perfect ring if and only if each pure
ideal of $R$ which is the annihilator of a finitely generated
ideal of $R$ is generated by an idempotent.
\end{thm}
\begin{proof} By \cite[Theorem 3.8]{c} and Theorem \ref{12}. \end{proof}

The next example gives a  $(1,1)$-perfect ring $R$ and  a
 multiplicative set $S$ of $R$ such that $S^{-1}R$ is not a
 $(1,1)$-perfect ring.

 \begin{exmp}\cite[Example 5.17]{c}\label{28}
 Let  $ R= \Z[ f,x_{1}, x_{2},...]$, with defining relations
 $fx_{i}(1-x_{j})=0$,  $1\leq i < j$  and $2fx_{i}=0$,  $1\leq i$.
 Put $S=\{f^{n} / n \geq 1\}$. Then $R$ is an $(1,1)$-perfect ring,
 but $S^{-1}R$ is not a  $(1,1)$-perfect ring.

 \end{exmp}
\end{section}
%%%%%%%%%%%%%%%%%%%%%%%%%%%%%%%%%%%%%%%%%%%%%%%%%%%%%%%%
%%%%%%%%%%%%%%%%%%%%%%%%%%%%%%%%%%%%%%%%%%%%%%%%%%%%%%%%%%%
%%%%%%%%%%%%%%%%%%%%%%%%%%%%%%%%%%%%%%%%%%%%%%%%%%%%%%%%
%%%%%%%%%%%%%%%%%%%%Section 2.%%%%%%%%%%%%%%%%%%%%%%%%%%%%
\begin{section}{Transfer of the $(n,d)$-perfect property in
pullbacks}\bigskip

 \quad  Pullbacks occupy an important niche in
 homological algebra because they produce
 interesting examples (see for example \cite[Section 1, Chapter
 5 ]{G}). For a history of pullbacks, see \cite[Appendix 2]{Gi} and
 \cite[pages 582-584]{Go}.\\

  The following  Theorem is the main result of this section.
\begin{thm}\label{5}
Let $A\hookrightarrow B$ be an injective flat ring homomorphism
and let $Q$ be a pure ideal of $A$ such that $QB=Q$ and
$\lambda_{A}(Q) \geq n-1$.
\begin{enumerate}
\item Assume that $B$ is an $(n,d)$-perfect ring. Then $A/Q$ is an $(n,d)$-perfect ring
if and only if $A$ is an $(n,d)$-perfect ring.
\item Assume that $B=S^{-1}A$, where $S$ is a multiplicative set of $A$.
 Then $A$ is an $(n,d)$-perfect ring if and only if $B$  and $A/Q$
are an $(n,d)$-perfect rings.
\end{enumerate}
\end{thm}\bigskip

Before proving this Theorem, we establish the following
Lemmas.\bigskip

  At the start, we recall the notion of flat epimorphism of
 rings, which is defined as
follows: Let
 $\Phi:A\rightarrow B$ be a ring homomorphism. $B$ $($ or $\Phi)$ is
 called  a flat epimorphism of $A$,  if $B$ is a flat $A$-module and
 $\Phi$ is an epimorphism, that is, for any two ring homomorphism
 $B\overset{f}{\underset{g}{\rightrightarrows}}C$, $C$ a ring,
 satisfying $f\circ \Phi= g\circ\Phi$, we have $f=g$ \cite[pages
 13-14]{G}.
  For example $S^{-1}A$ is a flat epimorphism of $A$ for every
  multiplicative set $S$ of $A$. Also, the quotient ring $A/I$ is a
  flat epimorphism of $A$ for every pure ideal $I$ of $A$, that is,
  $A/I$ is a flat $A$-module \cite[Theorem 1. 2. 15]{G}.

\begin{lem}\label{7}
Let $A$ and $B$ be two rings such that  $\Phi: A\rightarrow B$ be
a flat epimorphism of  $A$ and  $\lambda_{A}(B)\geq n$.
\\ If $A$ is an $(n,d)$-perfect ring then $B$ is an $(n,d)$-perfect ring.
\\ In particular, if $A$ is an $(n,d)$-perfect ring, then so is the quotient
ring $A/I$ for every pure ideal $I$ of $A$  such that
$\lambda_{A}(I)\geq n-1$ $(n \geq 0)$.
\end{lem}

\begin{proof} Let $M$ be a $B$-module such that $\lambda_{B}(M)\geq n$
and $\fd_{B}(M)\leq d$. Our aim is to show that $\pd_{B}(M)\leq
d$.
\\ By hypothesis we have $\Tor^{A}_{k}(M,B)=0$ for all $k>0$.
  By  \cite[Proposition 4.1.3]{CE}, we have for any $B$-module $N$
$$(\ast) \qquad  \Ext_{A}^{d+1}(M,N\otimes_{A}B)\cong
\Ext_{B}^{d+1}(M\otimes_{A}B,N\otimes_{A}B)$$ From \cite[Theorem
1.2.19]{G} we get  $M\otimes_{A}B\cong M$ and $N\otimes_{A}B\cong
N$.
\\On the other hand, $\fd_{A}(M)\leq \fd_{B}(M)\leq d $ \cite[Exercise 10.p.123]{CE} and
$\lambda_{A}(M) \geq n$ \cite[Lemma 2.6]{DKM} . Thus
$\pd_{A}(M)\leq d$ since $A$ is an $(n,d)$-perfect ring. Hence,
$(\ast)$ implies $\Ext_{B}^{d+1}(M,N)=0$, therefore
$\pd_{B}(M)\leq d$.\end{proof}

The Example \ref{28} shows that Lemma \ref{7} is not true in
general without assuming that  $\lambda_{A}(B)\geq n$.

\begin{lem}\label{6}

Let $A\hookrightarrow B$ be an injective flat ring homomorphism
and let $Q$ be a pure ideal of $A$ such that $QB=Q$. Let $E$ be an
$A$-module. Then:

\begin{enumerate}
\item $\lambda_{A}(E) \geq n \Leftrightarrow \lambda_{B}(E\otimes
_{A}B) \geq n$ and $ \lambda_{A/Q}(E\otimes_{A}A/Q)\geq n$.

\item $\fd_{A}(E)\leq d \Leftrightarrow
\fd_{B}(E\otimes_{A}B)\leq d  $ and $ \fd
_{A/Q}(E\otimes_{A}A/Q)\leq d$.
\item $\pd_{A}(E)\leq d \Leftrightarrow \pd_{B}(E\otimes_{A}B)\leq d $
and $ \pd _{A/Q}(E\otimes_{A}A/Q)\leq d$.
\end{enumerate}
\end{lem}
 \begin{proof} Similar to that of \cite[Lemma 2.4]{DKMn} will be
 omitted.\end{proof}

\begin{proof}[Proof of Theorem \ref{5}] $1)$ If $A$ is an
 $(n,d)$-perfect ring since $Q$ is an $(n-1)$-presented pure ideal
 of $A$, by Lemma \ref{7}. $A/Q$ is an $(n,d)$-perfect
ring. \\ Conversely, assume that $B$ and $A/Q$ are an
$(n,d)$-perfect rings. Let $M$ be an $A$-module such that
$\lambda_{A}(M)\geq n$ and $\fd_{A}(M)\leq d$.
\\Then  $\lambda_{B}(M\otimes_{A}B)\geq n$ and $\fd_{B}(M\otimes_{A}B)\leq d $. So $\pd_{B}(M\otimes_{A}B)\leq d
$ since $B$ is an $(n,d)$-perfect ring. Also
$\lambda_{A/Q}(M\otimes_{A}A/Q)\geq n$
 and $ \fd
_{A/Q}(M\otimes_{A}A/Q)\leq d$ .  So  $ \pd
_{A/Q}(M\otimes_{A}A/Q)\leq d$  since  $A/Q$ is an $(n,d)$-perfect
rings. By Lemma \ref{6}. $\pd_{A}(M)\leq n$.  Therefore $A$ is an
$(n,d)$-perfect ring.
\\$2)$ Follows from Lemma \ref{7}. Lemma \ref{6} and $1)$.\end{proof}

\begin{cor}
 Let $D$ be an integral domain, $K=\qf(D)$ and let $n\geq 2$, $n\geq 0$  and $d\geq 0$ be
a positive integers. Consider the quotient ring
$S=K[X]/(X^{n}-X)=K+ \overline{X}K[\overline{X}] = K +
  I$ with $I= \overline{X}K[\overline{X}]$.
  Set $R=D + I$. Then $R$ is an $(m,d)$-perfect ring
if and only if $D$ is an $(m,d)$-perfect ring.

\end{cor}

\begin{proof} First we show that $I$ is a pure ideal of $R$. Let $u := \overline{X}^{i} ( a_{0} +
a_{1}\overline{X}+ ... + a_{n-1}\overline{X}^{n-1})$ be an element
of $I$,  where $a_{i}\in K$ for $1\leq i \leq n-1$, and $a_{0}\neq
0$. Hence $u( 1- \overline{X}^{n-1})= 0$  $(\ast)$ since
$\overline{X}^{i}(1-\overline{X}^{n-1})=\overline{X}^{i}-\overline{X}^{n+(i-1)}=\overline{X}^{i}-
\overline{X}^{n}\overline{X}^{i-1}=\overline{X}^{i}-\overline{X}^{i}=\overline{0}$.
Therefore,
$I$ is a pure ideal of $R$ by \cite[Theorem 1.2.15]{G} since $\overline{X}^{n-1} \in I$. \\
Our aim is to show that $\lambda_{R}(I) =\infty $. We have
$R\overline{X}^{n-1}
=(D+\overline{X}K[\overline{X}])\overline{X}^{n-1}
=D\overline{X}^{n-1}+\overline{X}K[\overline{X}]
=D\overline{X}^{n-1}+I =I$ since $\overline{X}^{n} =\overline{X}$.\\
We claim that $Ann_{R}(I) =R(1-\overline{X}^{n-1})$.  Indeed, by
$(\ast)$ $R(1-\overline{X}^{n-1}) \subseteq Ann_{R}(I)$.
Conversely, let $v := d + a_{1}\overline{X}+ ... +
a_{n-1}\overline{X}^{n-1} \in Ann_{R}(I)$, where $d \in D$ and
$a_{i}\in K$. Hence $0 =(d + a_{1}\overline{X}+ ... +
a_{n-1}\overline{X}^{n-1})\overline{X}^{n-1} = a_{1}\overline{X}+
... + (d+a_{n-1})\overline{X}^{n-1}$ and so $a_{1} =a_{2} =\ldots
=a_{n-2} =0$ and $d+a_{n-1} =0$. This means,
$v =d(1-\overline{X}^{n-1}) \in R(1-\overline{X}^{n-1})$ as desired. \\
Also, the same proof as above shows that
$Ann_{R}(R(1-\overline{X}^{n-1})) =I$.
Therefore, $\lambda_{R}(I) =\infty $ as desired.  \\
On the other hand, we have $S$ an artinian ring. According to
\cite[Corollary 28.8]{A} we deduce $S$ is a perfect ring, so $S$
is an $(m,d)$-perfect ring.
 \\We obtain the result by Theorem
\ref{5} and this completes the proof of Corollary 4.4. \end{proof}

From this  Corollary  we deduce easily the following example.

\begin{exmp}
 Let $D$ be an integral domain such that $\gldim(D)=d$, \\let
$K=\qf(D)$ and let $n\geq 2$.
  \\ Consider the quotient ring  $S=K[X]/(X^{n}-X)=K+ \overline{X}K[\overline{X}] = K +
  I$ with $I= \overline{X}K[\overline{X}]$.
  Set $R=D + I$.
   Then $R$ is a  $(1,d)$-perfect ring.

\end{exmp}
\begin{proof} Clear by Corollary 4.4 since $D$ is an $(1,d)$-perfect ring
(since  $\gldim(D) =d$).\end{proof}
\end{section}

%%%%%%%%%%%%%%%%%%%%%%%%%%%%%%%%%%%%%%%%%%%%%%%%%%%%%%%%%%%%%%%%%%%%%%%%%%%%%%%%%%%%%%%%%%%%%%%%%%%%%%%%%%%%%%%%%%%%
%%%%%%%%%%%%%%%%%%%%%%%%%%%%%%%%%%%%%%%%%%%%%%%%%%%%%%%%%%%%%%%%%%%%%%%%%%%%%%%%%%%%%%%%%%%%%%%%%%%%%%%%%%%%%%%%%%%%
%%%%%%%%%%%%%%%%%%%%%%%%%%%%%%%%%%%%%%%%%%%%%%%%%%%%%%%%%%%%%%%%%%%%%%%%%%%%%%%%%%%%%%%%%%%%%%%%%%%%%%%%%%%%%%%%%%%%
%%%%%%%%%%%%%%%%%%%%%%%%%%%%%%%%%%%%%%%%%%%%%%%%%%%%%%%%%%%%%%%%%%%%%%%%%%%%%%%%%%%%%%%%%%%%%%%%%%%%%%%%%%%%%%%%%%%%

\end{document}